\documentclass{amsart}
\usepackage{amssymb,latexsym,amsfonts}
\def\C{\mathbb C}
\def\R{\mathbb R}

\def\N{\mathbb N}
\def\re{\operatorname{Re}}
\def\im{\operatorname{Im}}
\newtheorem{thm}{Theorem}[section]
\newtheorem{lem}{Lemma}[section]

\begin{document}
\title{The escaping set of a quasiregular mapping}
\author{Walter Bergweiler}
\address{Mathematisches Seminar
\\Christian--Albrechts--Universit\"at zu Kiel
\\Ludewig--Meyn--Str.\ 4, D--24098 Kiel, Germany}
\email{bergweiler@math.uni-kiel.de}
\author{Alastair Fletcher}
\address{School of Mathematical Sciences
\\University of Nottingham
\\NG7 2RD, UK}
\email{alastair.fletcher@nottingham.ac.uk}
\author{Jim Langley}
\address{School of Mathematical Sciences
\\University of Nottingham
\\NG7 2RD, UK}
\email{jkl@maths.nott.ac.uk}
\author{Janis Meyer}
\address{School of Mathematical Sciences
\\University of Nottingham
\\NG7 2RD, UK}
\email{janis.meyer@maths.nottingham.ac.uk}
\thanks{This research was supported by:
the G.I.F, the German-Israeli
Foundation for Scientific Research and Development,
Grant G-809-234.6/2003, and the EU Research Training Network CODY
(Bergweiler); EPSRC grant RA22AP( Fletcher and Langley); the
ESF Research Networking Programme HCAA (Bergweiler and Langley);
DFG grant ME 3198/1-1 (Meyer).}

\begin{abstract}
We show that if the maximum modulus of a quasiregular mapping 
$f: \R^N \to \R^N$ grows sufficiently rapidly then there exists a
non-empty escaping set $I(f)$ consisting of
points whose forward orbits under iteration of $f$ tend to infinity.
We also construct a quasiregular mapping for which the closure of $I(f)$
has a bounded component. This stands
in contrast to the situation for entire functions in the complex
plane, for which all components of the closure of $I(f)$ are unbounded,
and where it is in fact conjectured that all components of $I(f)$ are
unbounded.
\end{abstract}
\maketitle
\section{Introduction}

In the study \cite{Ber4} of the dynamics of nonlinear entire functions
$f: \C \to \C$ considerable recent attention has focussed
on the escaping set
$$
I(f) = \{ z \in \C : \lim_{n \to \infty} f^n(z) = \infty \} ,
$$
where $f^1  = f, f^{n+1} = f \circ f^n$ denote the iterates of $f$.
Eremenko \cite{Er0} proved that if $f$ is transcendental then
$I(f) \neq \emptyset$ and indeed that, in keeping with the nonlinear
polynomial case \cite{Stei2}, the boundary of $I(f)$ is the Julia set
$J(f)$. The proof in \cite{Er0}
that $I(f)$ is non-empty
is based on the Wiman-Valiron theory \cite{Hay5}.

For transcendental entire
functions $f$, Eremenko went on to prove in \cite{Er0} that 
all components of 
the closure of $I(f)$ are unbounded, and to conjecture that the same is true
of $I(f)$ itself. For entire functions with bounded postcritical 
set this conjecture was proved by Rempe
\cite{rempe}, and for the general case
it was shown by Rippon and Stallard \cite{ripstall1}
that $I(f)$ has at least one unbounded component. 

In the meromorphic case the set $I(f)$ was first studied by Dominguez
\cite{Dom}, who proved that again $I(f) \neq \emptyset$ and 
$\partial I(f) = J(f)$. For meromorphic $f$ it is possible that all 
components of $I(f)$ are bounded \cite{Dom}, and the closure
of $I(f)$ may have bounded components even if $f$ has only one pole
\cite[p.229]{Dom}. On the other hand $I(f)$ always has at least one
unbounded component if the inverse function $f^{-1}$ has a direct
transcendental singularity over infinity: this
was proved by Bergweiler, Rippon and Stallard \cite{BRS} 
by developing an analogue of the Wiman-Valiron
theory in the presence of a direct singularity. 

The present paper is concerned with the escaping set for
quasiregular mappings $f: \R^N \to \R^N$ \cite{rick1},
which represent a natural
counterpart in higher real dimensions of analytic functions in the plane, 
and exhibit many analogous properties, a highlight among
these being Rickman's Picard theorem for entire quasiregular maps
\cite{rick2,rick1}. 
For the precise definition and further properties 
of quasiregular mappings we
refer the reader to Rickman's text \cite{rick1}.
Now the iterates of an entire quasiregular map are again quasiregular, 
and properties such as the existence of periodic points were investigated
in \cite{Ber6,siebert}. Further, there is increasing interest in the dynamics
of quasiregular mappings on the compactification $\overline{\R^N}$ of
$\R^N$, although attention has been restricted to
mappings which are uniformly quasiregular
in the sense that all iterates have a common 
bound on their dilatation: see  
\cite[Section 21]{IwMa} and \cite{HMM}. In the absence of this
uniform quasiregularity there are evidently some difficulties
in extending some concepts of complex dynamics to
quasiregular mappings in general, but the escaping set $I(f)$
makes sense nevertheless, and we 
shall prove the following theorem.

\begin{thm}\label{thm1}
Let $N \geq 2$ and $K > 1$. Then there
exists $J > 1$, depending only on $N$ and $K$, with the following
property.

Let $R > 0$ and let $f: D_R \to \R^N$ be a $K$-quasiregular 
mapping, where $D_R \subseteq \R^N$ is a domain containing the set
\begin{equation}
B_R = \{ x \in \R^N : R \leq |x| < \infty \} .
\label{4}
\end{equation}
Assume that $f$ satisfies
\begin{equation}
\liminf_{r \to \infty} \frac{M(r, f)}{r} \geq J , \quad \hbox{where} \quad
M(r, f) = \max \{ |f(x)| : |x| = r \}  ,
\label{2}
\end{equation}
and define the escaping set by
\begin{equation}
I(f) = \{ x \in \R^N : \lim_{n \to \infty} f^n(x)= \infty \},
\quad f^1 = f, \quad f^{n+1} = f \circ f^n .
\label{3}
\end{equation}
Then $I(f)$ is non-empty. 
If, in addition, $f$ is $K$-quasiregular on $\R^N$ 
then $I(f)$ has an unbounded component.
\end{thm}  

The proof of Theorem \ref{thm1} is based on the approach of Dominguez
\cite{Dom}, as well as that of Rippon and Stallard \cite{ripstall1}.
A key role is played also by the analogue
of Zalcman's lemma \cite{Zalcman1,Zalcman2} developed for quasiregular
mappings by Miniowitz \cite{Minio} (see \S\ref{qr}). 
It seems worth observing that in Theorem \ref{thm1} the hypothesis 
(\ref{2}) cannot be replaced by 
$$
\liminf_{r \to \infty} \frac{M(r,f)}{r} > 1,
$$
as the following example shows. Take cylindrical polar 
coordinates $ r \cos \theta ,  r \sin \theta ,  x_3$ in $\R^3$,
let $\lambda > 0$ and
and let $f$ be the mapping defined by
$$
0 \rightarrow 0, \quad ( r e^{i \theta } , x_3) \rightarrow 
(r e^{ \lambda \cos \theta + i ( \theta + \pi ) } , x_3) .
$$
Then $f^2$ is given by
$$
( r e^{i \theta } , x_3) \rightarrow 
(r e^{ \lambda  \cos \theta + \lambda \cos (\theta + \pi )
+ i ( \theta + 2 \pi ) } , x_3)
$$
and so is the identity, while since $f$ is $C^1$ on $\R^3 \setminus \{ 0 \}$
and satisfies $f(2x) = 2 f(x)$ it is easy to see that $f$ is
quasiconformal on $\R^3$. On the other hand if $f: \R^N \to \R^N$ is 
quasiregular
with an essential singularity at infinity, then $M(r, f)/r \to \infty$ as
$r \to \infty$ (see, for example, \cite[Lemma 3.4]{Ber6}) so that
(\ref{2}) holds with any $J > 1$.

Next, we show in \S\ref{exa} that there exists a quasiregular
mapping $f$ on $\R^2$ with an essential singularity at infinity, such
that the closure of $I(f)$ has a bounded component. Thus while 
the result of \cite{ripstall1} that $I(f)$ has at least one 
unbounded component extends to quasiregular mappings by
Theorem \ref{thm1}, Eremenko's theorem \cite{Er0} that all
components of the closure of $I(f)$ are unbounded does not.

We remark finally that it is easy to show that if
$f$ is quasimeromorphic with infinitely many poles in $\R^N$ then
$I(f)$ is non-empty, and for completeness we outline how this is proved
in \S\ref{mero}, using the ``jumping from pole to pole'' method
\cite{BRS,Dom}.

\section{Theorems of Rickman and Miniowitz}\label{qr}
Let $G$ be a domain in $\R^N$. A continuous mapping $f: G \to \R^N$ 
is called quasiregular \cite{rick1} if $f$ belongs to the Sobolev
space $W_{N, {loc}}^1 (G)$ and there exists $K \in [1, \infty)$ such
that
$$
|f'(x)|^N \leq K J_f \quad \hbox{a.e. in $G$.}
$$
Moreover, $f$ is called $K$-quasiregular if its inner and outer dilations
do not exceed $K$: 
for the details and equivalent definitions we refer the reader to
\cite{rick1}.  
Rickman proved \cite{rick2,rick1}
that given $N \geq 2$ and $K \geq 1$ there exists an
integer $C(N, K)$ such that if $f$ is $K$-quasiregular on $\R^N$
and omits $C(N, K)$ distinct values $a_j \in \R^N$ then $f$ is constant.
Here $C(2, K) = 2$ because a quasiregular mapping in $\R^2$ may be
written as the composition of a quasiconformal mapping with an entire
function, but for $N \geq 3$ this integer $C(N, K)$ in general
exceeds $2$ \cite{rick3,rick1}.

Miniowitz \cite{Minio}
established for quasiregular mappings the following direct analogue 
of Zalcman's lemma \cite{Zalcman1,Zalcman2}. 
A family $F$ of $K$-quasiregular mappings on the 
unit ball $B^N$ of $\R^N$ is not normal if and only if there exist
$$
f_n \in F, \quad x_n \in B^N, \quad x_n \to \hat x \in B^N,
\quad \rho_n \to 0+ ,
$$
such that
$$
f_n(x_n + \rho_n x) \to f(x)
$$
locally uniformly in $\R^N$, where $f$ is $K$-quasiregular and non-constant.
Using this she established the following analogue of Montel's theorem,
in which $C(N, K)$ is the integer from Rickman's theorem
\cite{rick2} and $\chi (x, y)$ denotes the
spherical distance on $\R^N$. 

\begin{thm}[\cite{Minio}]\label{thmA}
Let $N \geq 2, K > 1, \varepsilon > 0$ and let $D$ be a domain in
$\R^N$. Let $F$ be a family of functions $K$-quasiregular on $D$ with
the following property. Each $f \in F$ omits 
$q = C(N, K)$ values $a_1 (f), \ldots , a_q(f)$ on $D$, which may depend
on $f$ but satisfy
$$
\chi (a_j(f), a_k(f)) \geq \varepsilon \quad \hbox{for} \quad j \neq k.
$$
Then $F$ is normal on $D$.
\end{thm}

Theorem \ref{thmA} leads at once to the following standard lemma of Schottky
type.

\begin{lem}\label{lem1}
Let $N \geq 2$ and $K > 1$. Then there exists $Q > 2$ with the following
property. Let $f$ be $K$-quasiregular on the set
$\{ x \in \R^N : 1 < |x| < 4 \}$ such that $f$ omits
$q = C(N, K)$ values $y_1, \ldots , y_q $, with
$$
|y_j| = 4^{j-1}, \quad j = 1, \ldots , q.
$$
If $\min \{ |f(x)| : |x| = 2 \} \leq 2$
then $\max \{ |f(x)| : |x| = 2 \} \leq Q$.
\end{lem}

\section{Two lemmas needed for Theorem \ref{thm1}}

We need the following two facts, the first of which is
from Newman's book \cite[Exercise, p.84]{Newman}:

\begin{lem}\label{lemA}
Let $G$ be a continuum in $\overline{\R^N} = \R^N \cup \{ \infty \}$
such that $\infty \in G$, and let $H$ be a component of $\R^N \cap G$.
Then $H$ is unbounded.
\end{lem}

This leads on to the second fact we need:

\begin{lem}\label{lem11}
Let $E$ be a continuum in
$\overline{\R^N} $ such that $\infty \in E$, and let $g : \R^N \to \R^N$
be a continuous open mapping. Then the preimage
$$g^{-1}(E)=
\{ x \in \R^N : g(x) \in E \} $$ 
cannot have a bounded component.
\end{lem}

For completeness we give a proof of Lemma \ref{lem11} in
\S\ref{pflem11}.

\section{An analogue of Bohr's theorem}\label{bohr}

Let $f: D_R \to \R^N$ be $K$-quasiregular,
where $D_R \subseteq \R^N$ is a domain containing
the set $B_R$ in (\ref{4}), and assume that $f$ satisfies
(\ref{2}) for some $J > 1$.
For $0 \leq r < s \leq \infty$ set
$$
A(r, s) = \{ x \in \R^N : r < |x| < s \} . 
$$
Using (\ref{2}) choose $s_0 > R$ such
that
$$
M(r, f) > M(R, f) \quad \hbox{for all} \quad r \geq s_0 .
$$
Then $M(r, f)$ is strictly increasing on $[s_0, \infty)$ because if
$s_0 \leq r_1 < r_2 < \infty$ and $M(r_2, f) \leq M(r_1 , f)$ then
$|f(x)|$ has a local maximum at some $\hat x \in A(R, r_2)$, which
contradicts the openness of the mapping $f$. Following Dominguez \cite{Dom}
we establish a lemma analogous to Bohr's theorem.

\begin{lem}\label{lem2}
Let $c = 1/2Q$, where $Q$ is the constant of Lemma \ref{lem1}.
Then for all sufficiently large $\rho$ there exists 
$L \geq c M( \rho /2 , f)$ such that
$$
S(0, L) = \{ x \in \R^N : |x| = L \} \subseteq f( A(R, \rho)) .
$$
\end{lem}
\begin{proof}
Using (\ref{2}) let $\rho$ be so large that
\begin{equation}
\rho > 4R \quad \hbox{and} \quad S = c M( \rho /2 , f) > 2T = 4 M(R, f),
\label{rhodef}
\end{equation}
and assume that the assertion of the lemma is false for $\rho$.
Then for $j = 1, \ldots , q$, where $q = C(N, K)$ is the
integer from Rickman's Picard theorem
\cite{rick2} (see \S\ref{qr}), there exists $a_j \in \R^N$ with
\begin{equation}
|a_j| = 4^{j-1} S\quad \hbox{and}  \quad a_j \not \in  f( A(R, \rho)) .
\label{ajdef}
\end{equation}
Furthermore, there exists
$x_1 \in A(R, \rho/2) $ such that $|f(x_1)| = S $.
To see this join a point $x_0$ on $S(0, \rho /2)$ such that
$|f(x_0)| = M( \rho /2, f)$ to $S(0, R)$ by a radial segment and use
(\ref{rhodef}) and the fact that $c < 1$. 
Let $G$ be the component
of the set 
$$\{ x \in \R^N : T < |f(x)| < 2S \} $$
which contains $x_1$.
Then $G \subseteq A(R, \infty)$ by (\ref{rhodef}).
Suppose first that $G \subseteq A(R, \rho/2)$. Then the closure 
$\overline{G}$ of $G$ lies in $A(R, \rho)$, by (\ref{rhodef}) again.
Choose a geodesic $\sigma \subseteq S(0, S)$ joining
$f(x_1)$ to $a_1$. Let
$$
\mu = \inf \{ |f(x) - a_1 | : x \in \overline{G}, f(x) \in \sigma \}
$$
and take $\zeta_n \in \overline{G}$ with $f( \zeta_n) \in \sigma$ 
and $|f(\zeta_n) - a_1 | \to \mu$. Then we may assume that
$\zeta_n \to \hat \zeta \in \overline{G}$, and we have
$f( \hat \zeta ) \in \sigma $ and so $\hat \zeta \in G$. But then the
open mapping theorem forces
$\mu = | f( \hat \zeta ) - a_1 | = 0$,
which contradicts (\ref{ajdef}).

Thus $G \not \subseteq A(R, \rho/2)$ and this implies using (\ref{rhodef})
again that there must exist
$x_2 $ on $ S(0, \rho/2)$ such that $ |f(x_2)| \leq 2S $.
By (\ref{rhodef}) and (\ref{ajdef}) the function
$g(x) = f(x \rho/4)/S$
is $K$-quasiregular on $A(1, 4)$,  and omits the $q$ values
$y_j = a_j/S$, which satisfy $|y_j| = 4^{j-1}$.
Since $|g(4 x_2/ \rho ) |\leq 2$, Lemma \ref{lem1} implies that 
$|g(x)| \leq Q$ for $|x| = 2$, which gives
$$
M( \rho/2, f) \leq QS = Qc M( \rho/2, f) = \frac{M(\rho/2,f)}2 ,
$$
a contradiction.
\end{proof}

\section{Proof of Theorem \ref{thm1}}

Again let $f: D_R \to \R^N$ be $K$-quasiregular,
where $D_R \subseteq \R^N$ is a domain containing
the set $B_R$ in (\ref{4}), but this time assume that $f$ satisfies
(\ref{2}) for some large positive $J$. Retain the notation of 
\S\ref{bohr}. 
Following Dominguez' method \cite{Dom}
let $\rho_0 > R$ be so large that 
every $\rho \geq \rho_0$ satisfies the conclusion of Lemma \ref{lem2} 
and further that, with the same constant $c$ as in Lemma \ref{lem2},
\begin{equation}
c M(\rho/2, f) > 4 \rho > \rho > M(R, f) \quad \hbox{for all} \quad 
\rho \geq \rho_0 ,
\label{rhodef2}
\end{equation}
which is possible by (\ref{2}) and the assumption that
$J$ is large. Fix  $\rho \geq \rho_0$. 

\begin{lem}\label{lem3}
There exist bounded open sets $G_0, G_1, \ldots $ 
with the following properties. \\
(i) The set $\overline{\R^N} \setminus G_n$ has two components, namely
$$
\widetilde G_n = \overline{B(0, R)} = \{ x \in \R^N : |x| \leq R \} 
$$
and $G_n^* = A_n $, which satisfies $\infty \in A_n$.\\
(ii) We have
\begin{equation}
 \{ x \in \R^N : R < |x| \leq  2^n \rho \} \subseteq G_n .
\label{g2}
\end{equation}
(iii)
The sets $G_n$, $A_n$ and 
$\gamma_n = \partial A_n $ satisfy
\begin{equation}
\gamma_{n+1} \subseteq f( \gamma_n) \quad
\hbox{and} \quad f(G_n) \cap A_{n+1} = \emptyset .
\label{g3}\end{equation}
\end{lem}
\begin{proof}
The open sets $G_n$ will be constructed inductively.
We begin by setting
$G_0 = A(R, \rho')$ for some $\rho' > \rho$,
so that (\ref{g2}) obviously is
satisfied for $n=0$. It remains to show how to construct
$G_{n+1}$ given the existence of $G_0, \ldots, G_n$ for some $n \geq 0$.
The fact that $f$ maps open sets to open sets gives
\begin{equation}
\partial f(G_n) \subseteq f( \partial G_n) 
= f(S(0,  R)) \cup f( \gamma_n ) ,
\label{8}
\end{equation}
using (i) and the definition $\gamma_n = \partial A_n$. 
By Lemma \ref{lem2}, (\ref{rhodef2}) and (\ref{g2}) there exists
\begin{equation}
T_n \geq c M( 2^{n-1} \rho , f) > 2^{n+2} \rho \quad \hbox{with}
\quad 
S(0, T_n) \subseteq f( A(R, 2^n \rho) ) \subseteq f(G_n).
\label{10}
\end{equation}
Now $f(G_n)$ is a bounded open set, so 
let $A_{n+1}$ be the component of $\overline{\R^N} \setminus f(G_n)$ which 
contains $\infty$ and set
\begin{equation}
\gamma_{n+1} = \partial A_{n+1} .
\label{11}
\end{equation}
Then by (\ref{10}) we have
\begin{equation}
\gamma_{n+1} \subseteq A_{n+1} \subseteq A(2^{n+2} \rho, \infty ), 
\label{12}
\end{equation}
and (\ref{rhodef2}), 
(\ref{8}) and (\ref{11}) 
imply the first assertion of 
(\ref{g3}). Let 
$$G_{n+1} = \R^N \setminus (\overline{B(0, R)} \cup A_{n+1}) .$$
Then (i) is satisfied with $n$ replaced by $n+1$, and
the second assertion of (\ref{g3}) follows from the definition of
$A_{n+1}$. Finally (\ref{12}) shows that
(\ref{g2}) is satisfied with $n$ replaced by $n+1$, and so the
induction is complete. 
\end{proof}

\begin{lem}\label{lem4}
Let $w \in \gamma_n$. Then there exists $z_n \in \gamma_0$ with
$f^n(z_n) = w$ and 
\begin{equation}
f^m(z_n) \in \gamma_m \quad \hbox{for} \quad 
m = 0, \ldots , n.
\label{14}
\end{equation}
\end{lem}
\begin{proof}
This is easily proved using induction and (\ref{g3}).
\end{proof}

Now take a sequence of points $z_n \in \gamma_0$ satisfying (\ref{14}).
We may assume that $(z_n)$ converges to $\hat z \in \gamma_0$, and
we have, by (\ref{14}),
\begin{equation}
f^m (\hat z) = \lim_{n\to \infty} f^m (z_n) \in \gamma_m 
\quad \hbox{ for each $m \geq 0$}.
\label{15}
\end{equation}
Using (\ref{12}) we get $\hat z \in I(f)$ and hence
$I(f)$ is non-empty. This proves the first assertion of Theorem \ref{thm1}.

The second assertion will be established by modifying the
method of Rippon and Stallard
\cite{ripstall1}, so assume that $f$ is $K$-quasiregular in
$\R^N$ and take $\hat z$ satisfying (\ref{15}).
As before let $A_n = G_n^*$ be the component of
$\overline{\R^N}\setminus G_n$ containing $\infty$,
and let $L_n$ be the component of $f^{-n} (A_n) $
containing $\hat z$, which is well-defined since $
f^n ( \hat z)  \in \gamma_n$ and
$\gamma_n = \partial A_n$ by definition.

\begin{lem}\label{lem5}
$L_n$ is closed and unbounded.
\end{lem}
\begin{proof} $L_n$ is closed since $A_n$ is closed, and 
$L_n$ is unbounded by Lemma \ref{lem11}.
\end{proof}

\begin{lem}\label{lem6}
We have $L_{n+1} \subseteq L_n$ for $n=0, 1, \ldots $.
\end{lem} 
\begin{proof} Suppose that $f^{n+1}(z') \in A_{n+1}$ but
$f^n(z') \not \in A_n$. Thus either $|f^n(z')| \leq R$ or
$f^n(z') \in G_n$, from which we obtain 
$f^{n+1} (z') \not \in A_{n+1}$, in the first case 
from (\ref{rhodef2}) and (\ref{g2}) and in the second case from
(\ref{g3}), and this is a contradiction.
Hence if $z' \in L_{n+1} $ then $z'$ lies in a component of
$f^{-n-1} (A_{n+1})$ which contains $\hat z$, and
this component in turn lies in a component
of $f^{-n} (A_n)$. Hence we get
$z' \in L_n$. 
\end{proof}

We may now write
$$
K_n = L_n \cup \{ \infty \} , \quad
\{ \hat z, \infty \} \subseteq K_{n+1} \subseteq K_n, \quad 
\{ \hat z, \infty \} \subseteq K = \bigcap_{n=0}^\infty K_n .
$$
Since $K_n$ is compact and connected so is
$K$ \cite[Theorem 5.3, p.81]{Newman}. 
Let $\Gamma$ be the component of $K \setminus \{ \infty \}$ which
contains $\hat z$. Then $\Gamma $ is unbounded by Lemma \ref{lemA}.
Now for $w \in \Gamma$ we have $w \in L_n$ and so 
$f^n (w) \in A_n = G_n^*$, so that $w \in I(f)$ by (\ref{g2}).
This completes the proof of Theorem \ref{thm1}.

We do not know whether the second conclusion of Theorem \ref{thm1} holds if
$f$ is only quasiregular on the set $B_R$ in (\ref{4}), but this
seems unlikely. The difficulty is
that for large $n$ we cannot control the behaviour of $f^n$ near
$S(0, R)$ and so the component $L_n$ in Lemma \ref{lem5}
may in principle be bounded. 

\section{A quasiregular mapping $f$ for which
$\overline{I(f)}$ has a bounded component}\label{exa}

To show that there exists a quasiregular mapping $f: \C \to \C$ such that
the closure of the escaping set $I(f)$ has a bounded component, we begin by
constructing a quasiconformal map $g$ with the following properties. 
For each $z$ in
the punctured disc $A: = \{ z \in \C : 0 < |z| < 1 \}$ the iterates
$g^n$ satisfy $\lim_{n \to \infty } | g^n(z)| = 1$, and 
we have $\lim_{n \to \infty} g^n(1/2) = 1$. On the other hand 
there exist annuli $A_n \subseteq A$ such that $g$ maps $A_n$ onto
$A_{n+1}$, but with sufficient rotation that for each
$z \in A_n$
infinitely many of the forward images $g^k(z)$ lie away from $1$.
A map $h$ is then obtained from $g$ by  
conjugation with a M\"obius map $L$ which sends $1$ to
$\infty$, and finally $h$ is interpolated on a sector to ensure that the
resulting function has an essential singularity at infinity.

We will use the fact that if $p$ is quasiregular on a domain $D \subseteq \C$
and
$$
p_z = \frac{\partial p}{\partial z} = \frac12
\left( \frac{\partial p}{\partial x} - i
\frac{\partial p}{\partial y} \right)
$$
is bounded below in modulus on $D$, and if $q$ is continuous and
such that the partial derivatives
$q_x, q_y$ are sufficiently small on $D$, then $p + q$ is quasiregular
on $D$. If $0 \notin D \cup p(D)$ the same property may be applied locally
to $\log p$ as
a function of $\log z$. 
 
Turning to the detailed construction, we define
$a:[1,2]\to [0, \pi/4] $ by
$$a(r)=\frac{\pi}{4}-\arcsin\left(\frac{\sqrt{2}}{2r}\right).$$
Then
an application of the sine rule shows that the line segment
$$\re z = 1 + \im z , \quad 1 \leq |z| \leq 2,$$
is parametrized by
$z = r e^{i a(r)} $.

For $c>0$ we define $g:\C\to\C$ as follows.
Let $g(0)=0$ and for $z=re^{it}$
with $r > 0$ and $- \pi \leq t \leq \pi $ set:
\begin{equation*}
g(z)=
\begin{cases} 
\frac{4}{3}r\exp\left(i\left(t+c\left|\sin t \right| \right) \right),$$
& \text{$0<r< \frac12$;} \\[2mm]
\frac{1}{2-r}\exp\left(i\left(t+c\left|\sin t
\right| + c(1-r)^2\left|\sin\left(\frac{\pi}{1-r}\right)\right|
\right) \right),
&\text{$\frac12 \leq r<1$;}\\[2mm]
r \exp\left(i\left(t+c(2-r)
\sin \left(\frac{|t|-a(r)}{\pi-a(r)}\pi\right) 
\right)\right),
&\text{$1\leq r\leq 2$, $a(r)<|t|$;}\\[2mm]
r\exp(it), 
&\text{$1\leq r\leq 2$, $|t|\leq a(r)$;}\\[2mm]
r\exp(it), 
&\text{$r>2$.}\\
\end{cases}
\end{equation*}
Then $g$ is continuous on $\C$. Moreover, 
if $c$ is sufficiently small then $g$ is quasiconformal,
and in particular we choose $c< \pi/4$.
Note that, by the choice of $a(r)$, 
\begin{equation}
g(z)=z \quad \hbox{if} \quad \re z\geq |\im z|+1.
\label{gid}
\end{equation}

For $n\in\N$ we have 
\begin{equation}
g\left(1-\frac{1}{n+1}\right)= 1-\frac{1}{n+2}.
\label{g1/2}
\end{equation}
For $n\in\N$, $n\geq 2$, we consider the annulus
$$A_n:=\left\{ z\in \C: 1-\frac{1}{n+1/4}<|z|<1-\frac{1}{n+3/4}
\right\}.$$
Then $g(A_n)=A_{n+1}$.

\begin{lem}\label{re<0}
For each $z\in A_n$ with $\re z>0$
there exists $k\in\N$ with $\re g^k(z)\leq 0$.
\end{lem}

\begin{proof}
Let $z \in A_m$ and suppose first that
$0< t:=\arg z< \pi/2$. Then 
\begin{equation}
\pi > t + \frac{\pi}2 > t + 2c \geq \arg g(z)
\geq 
t+c\sin t
\geq 
t+\frac{2c}{\pi}t
=\left(1+\frac{2c}{\pi}\right)t.
\label{t>0}
\end{equation}
On the other hand if $- \pi/2 <t=\arg z\leq 0$ then
\begin{equation}
\frac{\pi}2 > 
\arg g(z)\geq t+c|\sin t| + \frac{c\sqrt{2}}{2(m+3/4)^2}
\geq \left(1-\frac{2c}{\pi}\right)t +\frac{c'}{(m+1)^2} > - \frac{\pi}2  ,
\label{t<0}
\end{equation}
where $c':=\frac12 c\sqrt{2}$. In particular, 
(\ref{t>0}) and (\ref{t<0}) both hold with $\arg g(z)$ the
principal argument.

Suppose then that there exists $z\in A_n$ 
with $\re g^k(z) > 0$ for all integers $k \geq 0$, and set
$t_k = \arg g^k(z) \in ( - \pi /2, \pi /2)$. Then $g^k(z) \in A_{n+k}$.
If there exists
$k \geq 0$ with $0 < t_k < \pi /2$ then by repeated application of
(\ref{t>0}) we obtain $k' > k$ with $t_{k'} 
\in \left(\pi/2,\pi\right)$, a contradiction. Hence we must
have $- \pi /2 < t_k \leq 0$ for all $k\geq0$. But then repeated application
of (\ref{t<0}) gives, for large $k$,
$$
t_{k-1} \geq \left(1-\frac{2c}{\pi}\right)^{k-1} t_0 , \quad
t_k \geq \left(1-\frac{2c}{\pi}\right) t_{k-1} + \frac{c'}{(n+k)^2} 
> 0,
$$
again a contradiction. 
\end{proof}

With the M\"obius transformation 
$$L(z)=\frac{1}{1-z}$$
we now consider the map
$h:=L\circ g\circ L^{-1}$.
Then $h$ is a quasiconformal self-map of the plane.
Moreover, (\ref{gid}) gives
$h(z)=z$ if $\re L^{-1}(z)\geq |\im L^{-1}(z)|+1$,
which is equivalent to
$\re z\leq -|\im z|$, and we have
\begin{equation}
L(A_n) \subseteq \{ z \in \C : \re z > 0 \} \quad \hbox{and} \quad
h(L(A_n)) = L(A_{n+1}),
\label{hAn}
\end{equation}
using the fact that $g(A_n) = A_{n+1}$.

It follows from (\ref{g1/2}) that
\begin{equation}
h(n+1) = n+2 \quad \hbox{for} \quad n \in \N, 
\label{hn2}
\end{equation}
and we deduce at once that $2\in I(h)$.
Next we show that $L(A_n)\cap I(h)=\emptyset$ for every integer $n\geq 2$.
In fact, suppose that $n \geq 2$ and $u\in L(A_n)\cap I(h)$.
Then there exists $j_0\in\N$ such that $|h^j(u)|>1$ for 
$j\geq j_0$. Put $w:=h^{j_0}(u)$ and $m:=n+j_0$.
Then 
$L^{-1}(w)\in A_m$ by (\ref{hAn}), and Lemma \ref{re<0}
gives
$k\geq 0$ with $\re g^k\left(L^{-1}(w)\right)\leq 0$.
Since $|L(z)|\leq 1$ for $\re z\leq 0$ we deduce that
$$\left|h^{k+j_0}(u)\right|=\left|h^k(w)\right|
=\left|L\left(g^k\left(L^{-1}(w)\right)\right)\right|
\leq 1,$$
contradicting the choice of~$j_0$. Thus $L(A_n)\cap I(h)=\emptyset$.

Since $A_2$ separates $\frac12$ from $1$ it follows
that $2$ lies in the bounded component of the complement of $L(A_2)$, and we
deduce that the component of $\overline{I(h)}$ containing $2$ is bounded.

To construct a quasiregular map $f:\C\to\C$ with an essential singularity
at $\infty$ for which the closure of $I(f)$ has a bounded component we
put $f(z)=h(z)$ for $\re z\geq -|\im z|$ and
$f(z)=z+ d \exp\left(z^4\right)$ for $\re z\leq -|\im z|-1$, where
$d$ is a small positive constant. 
In the remaining region $\Omega$ we define $f$ by interpolation, using
$$
f(z) = z - d \phi (z),
\quad
\phi (z) = \left( \re z + | \im z | \right) \exp\left(z^4\right) 
\quad \hbox{for} \quad -1 < \re z + | \im z | < 0.
$$
Since $\exp\left(z^4\right)$ tends to $0$ rapidly as $z$ tends to
infinity in $\Omega$, it is then clear that the partial derivatives of 
$\phi$ are bounded on $\Omega$, so that $f$ is quasiregular
on $\Omega$ because $d$ is small.

In particular we have
$f(z) = h(z)$ for $\re z > 0$ and so it follows from
(\ref{hn2}) that $2\in I(f)$, whereas $L(A_n) \cap I(f)$ is again empty
using (\ref{hAn}). Thus the component of $\overline{I(f)}$ 
containing $2$ is bounded.

\section{The quasimeromorphic case}\label{mero}

Let $f$ be $K$-quasimeromorphic in the set $B_R$ defined in (\ref{4}),
with a sequence of poles tending to $\infty$, and set
$R_{-1} = R$. Choose
$x_j, D_j, R_j$ for $ j = 0, 1, 2,
\ldots $ as follows. Each $x_j$ is a pole of $f$, 
and $D_j$ is a bounded component of 
the set
$\{ x \in B_R : R_j < |f(x)| \leq \infty \}$ which contains $x_j$ but
no other pole of $f$, such that $D_j$ is mapped by $f$ onto
$\{ y \in \overline{\R^N} : R_j < |y| \leq \infty \}$.
Moreover, by choosing $R_{j+1}$ and $x_{j+1}$ sufficiently large, we may
ensure that
\begin{equation}
|x_{j+1}| > 4 R_j \quad \hbox{and} \quad  
D_{j+1} \subseteq \{ x \in \R^N : 2 R_j < |x| < \infty \}
\quad \hbox{for} \quad  j \geq -1.
\label{m1}
\end{equation}
Since
$|f(x)| = R_j $ for all $ x \in \partial D_j $
we may write, for $j \geq 0$, using (\ref{m1}),
\begin{equation}
C_j = \{ x \in D_j : f(x) \in D_{j+1} \} \subseteq
\overline{C_j} \subseteq D_j .
\label{m2}
\end{equation}
Now set 
\begin{equation}
X_0 = \overline{C_0}, \quad
X_{j+1} = \{ x \in X_j : f^{j+1}(x) \in \overline{C_{j+1}} \}.
\label{m5}
\end{equation}
Evidently $X_0$ is compact. Assuming that $X_j$ is compact,
it then follows that $X_{j+1}$ is the intersection of a compact set
with the closed set $f^{-j-1} ( \overline{C_{j+1}} )$ and so is compact.
Hence the $X_j$ form a nested sequence of compact sets.
We assert that
\begin{equation}
f^j(X_j) = \overline{C_j} .
\label{m6}
\end{equation}
We clearly have $f^j(X_j) \subseteq  \overline{C_j} $ by (\ref{m5}),
and (\ref{m6}) is
obviously true for $j=0$, so assume the assertion for some
$j \geq 0$ and take $w \in \overline{C_{j+1}} $. 
Since $f$ maps $D_j$ onto 
$\{ y \in \overline{\R^N} : R_j < |y| \leq \infty \}$, it follows
from (\ref{m1}) and (\ref{m2}) that there exists
$v \in C_j$ with $f(v) = w$. Hence there exists $x \in X_j$ with
$f^j(x) = v$ and $f^{j+1}(x) = w$, completing the induction.

Again since $f$ maps $D_j$ onto 
$\{ y \in \overline{\R^N} : R_j < |y| \leq \infty \}$, 
we evidently have $C_j \neq \emptyset$ and so $X_j $ is non-empty by 
(\ref{m6}). Hence there exists $x$ lying in the intersection of
the $X_j$, so that
$f^j(x) \in \overline{C_j} $ and
$x \in I(f)$ by (\ref{m1}) and (\ref{m2}).

\section{Proof of Lemma \ref{lem11}}\label{pflem11}

To establish
Lemma \ref{lem11} let $ E$ and $g$ be as in the statement
and assume that $g^{-1}(E)$ is non-empty since otherwise there is nothing to
prove. 
Note first that
$g^{-1}(E)$ is a closed subset of $\R^N$ by continuity.
Thus
$$
F =  g^{-1}(E) \cup \{ \infty \}
$$
is a compact subset of $\overline{\R^N}$. In order to prove Lemma \ref{lem11}
it therefore suffices in view of Lemma \ref{lemA} to show that $F$ is
connected. Suppose that this is not the case.
Then there is a partition of $F$ into non-empty
disjoint relatively closed (and so closed) sets $H_1, H_2$ such that
$\infty \in H_2$.
Let
$
W = \R^N \setminus H_2$.
Then $W$ is an open subset of $\R^N$, and $g(W \setminus H_1) \cap E =
\emptyset$.
Moreover, $H_1$ is a closed subset of $\overline{\R^N}$ and so compact, and
hence a compact subset of $\R^N$ since $\infty \in H_2$. 
Thus $g(H_1)$ is compact and so a non-empty closed subset
of $E$. 

Now suppose that there exist
$y_n \in E \setminus g(H_1)$ with $y_n \to \widetilde y \not
\in E \setminus g(H_1)$. Since $E$ is compact we 
have $\widetilde y \in E$ and so $\widetilde y \in g(H_1)$. 
Hence there exists $\widetilde x \in H_1 $ with $g( \widetilde x ) =
\widetilde y $ and for large enough $n$ there exists $x_n$ close
to $\widetilde x$ with $g(x_n) = y_n \in E \setminus g(H_1)$. 
But then we must have
$x_n \in H_1$, since $g(W \setminus H_1) \cap E = \emptyset$, and this is a
contradiction. So $E \setminus g(H_1)$ is also closed, but evidently
non-empty since $g( \R^N ) \subseteq \R^N $ and $\infty \in E$, 
which contradicts the hypothesis that $E$ is
connected.

\end{document}